\documentclass{birkau}

\usepackage[all]{xy}
\usepackage{tikz}
\usetikzlibrary{matrix,arrows}
\usepackage[mathscr]{eucal}
\usepackage{amssymb}
\usepackage{enumerate,graphicx}
\usepackage{verbatim,url}

\usepackage{diagrams}

\theoremstyle{plain}

 \newtheorem{thm}{Theorem}[section]
 \newtheorem{lem}[thm]{Lemma}
 \newtheorem{cor}[thm]{Corollary}
 \newtheorem{prop}[thm]{Proposition}
 
 \theoremstyle{definition}
 \newtheorem{df}[thm]{Definition}
 
 \newtheorem{ex}[thm]{Example}
 \newtheorem*{prob*}{Problem}

\newenvironment{newlist}
   {\begin{list}{}{\setlength{\labelsep}{0.25cm}
                   \setlength{\labelwidth}{0.65cm}
                   \setlength{\leftmargin}{0.9cm}}}
   {\end{list}}

\usepackage{pgf,tikz}
\usetikzlibrary{calc,positioning,shapes,arrows.meta}

\tikzset{%
 shaded/.style={draw, shape=circle, fill=black!35, inner sep=1.4pt},
 unshaded/.style={draw, shape=circle, fill=white, inner sep=1.4pt},
 quasi/.style={draw, shape=rectangle, rounded corners=3pt, fill=white, inner sep=2.5pt, minimum height=14.5pt},
 blob/.style={draw, shape=rectangle, rounded corners=12pt, thin, densely dotted},
 arrow/.style={->, thin, >=latex, shorten >=2.5pt, shorten <=2.5pt},
 order/.style={thin},
 curvy/.style={thin, looseness=1.2, bend angle=70},
 fatcurvy/.style={thin, looseness=1.7, bend angle=75},
 label/.style={shape=rectangle, inner sep=6pt},
 auto}

\newcommand{\lincol}{black}
\newcommand{\linth}{thick}
\newcommand{\po}[2][\pocol]{\filldraw[#1](#2) circle (2 pt);}

\newcommand{\li}[1]{\draw[\linth,\lincol] #1;}
\newcommand{\dotli}[1]{\draw[dotted,\linth,\lincol] #1;}
\newcommand{\smallpo}[2][\pocol]{\filldraw[#1](#2) circle (1.5 pt);}

\newcommand{\cat}[1]{\boldsymbol{\mathscr{#1}}}

\newcommand{\CL}{\cat L}

\font\bmi=cmmi8 scaled 1440
\newcommand{\powerset}{\raise.6ex\hbox{\bmi\char'175 }}

\newcommand{\mph}[2]{\CL^{\rm mp}(#1,#2)}

\newcommand{\D}[1]{\mathrm D(#1)}

\newcommand{\bbar}[1]{{\underline{\mathbf{#1}}}}
\newcommand{\twiddle}[1]{{\smash{\underset{\raise.375ex\hbox{$\smash\sim$}}
       {\mathbf{#1}}}\vphantom{\underline{\mathbf{#1}}}}}

\newcommand{\stwiddle}[1]{\smash{\underset{\smash{\raise.1ex\hbox{\small$\sim$}}}
                         {\mathbf{#1}}}\vphantom{#1}}

\newcommand{\twoB}{\bbar 2}

\newcommand{\T}{\mathscr{T}}

\newcommand{\C}{\mathbf C}

\newcommand{\Lalg}{\mathbf L}

\newcommand{\F}{\mathbf F}
\newcommand{\G}{\mathbf G}

\newcommand{\X}{\mathbf X}
\newcommand{\Y}{\mathbf Y}

\DeclareMathOperator{\gr}{gr}

\DeclareMathOperator{\Filt}{Filt}
\DeclareMathOperator{\Idl}{Idl}

\newcommand{\bigand}{\mathop{\bigwedge\kern -8.5truept \bigwedge}}
\newcommand{\Bigand}{\mathop{\bigwedge\kern -10truept \bigwedge}}
\newcommand{\littleand}{\mathbin{\wedge\kern -8truept \wedge}}
\newcommand{\bigor}{\mathop{\bigvee\kern -8.5truept \bigvee}}
\newcommand{\Bigor}{\mathop{\bigvee\kern -10truept \bigvee}}
\newcommand{\littleor}{\mathbin{\vee\kern -8truept \vee}}
\renewcommand{\le}{\leqslant}
\renewcommand{\leq}{\leqslant}
\renewcommand{\nleq}{\nleqslant}

\newcommand{\jty}{J^{\infty}}
\newcommand{\mty}{M^{\infty}}
\newcommand{\RTR}{R_{\triangleright}}
\newcommand{\RTL}{R_{\triangleleft}}

\begin{document}


\title[Canonical extensions of lattices]{Canonical extensions of lattices are more than
perfect}

\author[A.P.K. Craig]{Andrew P. K. Craig}
\address{Department of Mathematics and Applied Mathematics\\
University of Johannesburg\\PO Box 524, Auckland Park, 2006\\South~Africa}
\email{acraig@uj.ac.za}

\author[M.J. Gouveia]{Maria J. Gouveia}
\email{mjgouveia@fc.ul.pt}
\address{Faculdade de Ci\^encias da Universidade de Lisboa\\ P-1749-016 Lisboa\\ Portugal}

\author[M. Haviar]{Miroslav Haviar}
\address{Department of Mathematics\\Faculty of Natural Sciences, M. Bel University\\Tajovsk\'eho 40, 974~01 Bansk\'a Bystrica\\Slovakia}
\email{miroslav.haviar@umb.sk}

\dedicatory{In memory of Bjarni J\'onsson}

\subjclass{06B23, 06D50, 06B15}
\keywords{bounded lattice, canonical extension, perfect lattice, RS
frame, TiRS frame, TiRS graph, PTi lattice}

\begin{abstract}
In \cite{CGH15} we introduced TiRS graphs and TiRS frames to create
a new natural setting for duals of canonical extensions of lattices.
In this continuation of \cite{CGH15} we answer Problem 2 from there
by characterising the perfect lattices that are dual to TiRS frames
(and hence TiRS graphs). We introduce a new subclass of perfect
lattices called PTi lattices and show that the canonical extensions
of lattices are PTi lattices, and so are `more' than just perfect
lattices. We introduce morphisms of TiRS structures and put our
correspondence between TiRS graphs and TiRS frames from \cite{CGH15}
into a full categorical framework. We illustrate our correspondences
between classes of perfects lattices and classes of TiRS graphs by
examples.
\end{abstract}

\maketitle


\section{Introduction}\label{sec:intro}

An important aspect of  the study of 
lattice-based algebras in recent
decades has been the theory of canonical extensions.
This has its origins in the 
1951--52 papers of J\'{o}nsson and Tarski~\cite{JT51}. 
We refer to Gehrke and Vosmaer~\cite{GV11} for a survey of the theory of canonical
extensions for lattice-based algebras 
and, for further background,
to recent papers by
Gehrke~\cite{Ge18} and 
Goldblatt~\cite{Go18} and the references there,
in particular to
the first section of~\cite{Go18} 
called ``A~biography of canonical extension".

The canonical extensions of general (bounded) lattices were first
introduced by Gehrke and Harding~\cite{GH01} as the complete
lattices of Galois-closed sets associated with a polarity between
the filter lattice and the ideal lattice of the given lattice. 
(The same polarity was also used in the lattice representation of Hartonas and Dunn~\cite{HD97}.)
A~new
construction of the canonical extension of a general lattice was
provided in~\cite{CHP11} where it was based on a topological
representation of lattices by Plo\v{s}\v{c}ica~\cite{Pl95}. The
Plo\v{s}\v{c}ica representation presented a well-known
representation of general lattices due to Urquhart~\cite{U78} in the
spirit of the theory of natural dualities of Clark and Davey
\cite{CD98}. It used maximal partial maps into the two-element set
to represent elements of the first and second duals of a given
lattice.

An another construction of the canonical extensions of general
lattices was presented in \cite{CGH15} where Plo\v{s}\v{c}ica's
topological representation was used in tandem with Gehrke's
representation of perfect lattices via RS frames. (For the latter
we refer to papers \cite{DGP05} by 
Dunn, 
Gehrke and Palmigiano and
\cite{Ge06} by Gehrke.) In \cite{CGH15} we
also demonstrated a one-to-one correspondence between TiRS frames forming
a 
subclass of the RS frames
and TiRS graphs which we introduced as an abstraction of the duals
of general lattices in the Plo\v{s}\v{c}ica representation. This has
led to a new dual representation of the class of all finite lattices
via finite TiRS frames, or equivalently finite TiRS graphs, which
generalises the well-known Birkhoff dual representation between
finite distributive lattices and finite posets from the 1930s. (Here
we remark that every poset is a TiRS graph.) We use a~common concept
of \emph{TiRS structures} when we refer to both TiRS graphs and TiRS
frames without distinguishing between the two classes.

This paper has two goals:
\begin{enumerate}
\item To describe the additional properties that perfect lattices dual
to TiRS structures possess. This was listed as ``Problem 2''
in~\cite{CGH15}.
\item To describe the appropriate morphisms of TiRS structures and hence
to extend the one-to-one correspondence between the TiRS structures
from~\cite{CGH15} into a full categorical framework.
\end{enumerate}
We also show that the canonical extensions of lattices are PTi
lattices, which follows from their construction in~\cite{CGH15}
using Plo\v{s}\v{c}ica's and Gehrke's representations in tandem. We
present an example of a perfect but not PTi lattice together with
its dual TiRS graph and an example of a PTi lattice that is not the
canonical extension of any lattice together with its dual TiRS
frame.

\section{Preliminaries}


For a bounded lattice $\Lalg$, a \emph{completion} of $\Lalg$ is
defined to be a pair $(e,\C)$ where $\C$ is a complete lattice and
$e \colon \Lalg \hookrightarrow \C$ is an embedding. By a
\emph{filter element} (\emph{ideal element}) of  a completion
$(e,\C)$ of a bounded lattice~$\Lalg$ we mean an element of $\C$
which is a meet (join) of elements from $e(\Lalg)$. By $\mathbb{F}
(\C)$ and $\mathbb{I} (\C)$ are denoted the sets of all filter and
ideal elements of $\C$, respectively. (We remark that in the older
literature the filter (ideal) elements had been called closed (open)
elements.) A completion $(e,\C)$ of a bounded lattice $\Lalg$ is
called \emph{dense} if every element of~$\C$ can be expressed as
both a join of meets and a meet of joins of elements from
$e(\Lalg)$.  A completion $(e,\C)$ of $\Lalg$ is called
\emph{compact} if, for any sets $A \subseteq \mathbb{F} (\C)$ and
$B\subseteq \mathbb{I} (\C)$ with $\bigwedge A \le \bigvee B$, there
exist finite subsets $A' \subseteq A$ and $B' \subseteq B$ such that
$\bigwedge A' \le \bigvee B'$. (We remark that the sets $A, B$ in
the definition of compactness above can alternatively be taken as
arbitrary subsets of $L$.)

Gehrke and Harding \cite{GH01} defined abstractly \emph{the
canonical extension} $\Lalg^\delta $ of a general bounded lattice
$\Lalg$ as a dense and compact completion of $\Lalg$.  They proved
that every bounded lattice $\Lalg$ has a canonical extension and
that it is unique up to an  isomorphism that fixes the elements
of~$\Lalg$. Concretely, they constructed $\Lalg^\delta $ as the
complete lattice of Galois-stable sets of the polarity $R$ between
the filter lattice $\Filt (\Lalg)$ and the ideal lattice
$\Idl (\Lalg)$ of $\Lalg$ where the polarity is given by
$(F,I)\in R$ if $\ F \cap I \ne \emptyset$.

A filter-ideal pair $(F,I)$ will be called \emph{maximal} if $F$ and $I$ are maximal with respect to being disjoint from one another.
In our final section we shall use the following result from~\cite{GH01}:

\begin{lem}[{\cite[Lemma 3.4]{GH01}}]\label{GH-Lem3.4}
Let $(e,\C)$ be a canonical extension of $\Lalg$.
\begin{itemize}
\item[(1)] $x\in J^{\infty}(\C)$ if and only if $x=\bigwedge e[F]$ for some maximal pair $(F,I)$ of\, $\Lalg$;
\item[(2)] $x\in M^{\infty}(\C)$ if and only if $x=\bigvee e[I]$ for some maximal pair $(F,I)$ of $\Lalg$.
\end{itemize}
Further, each element of $\C$ is a join of completely join irreducibles and a meet
of completely meet irreducibles.
\end{lem}

Plo\v{s}\v{c}ica's dual~\cite[Section 1]{Pl95}
of a
bounded lattice $\Lalg$ is a graph with topology, $\D{\Lalg} =
(\mph{\Lalg}{\twoB},E,\T)$, where $\mph{\Lalg}{\twoB}$ is the set of
maximal partial homomorphisms from $\Lalg$ into $\twoB$. The graph
relation $E$ is defined by
$$(f,g) \in E\quad \text{if} \quad (\forall\, a \in dom\, f \cap dom\,
g) \ f(a)\leqslant g(a),
$$
or equivalently,
$$(f,g) \in E  \quad \text{if} \quad f^{-1}(1) \cap g^{-1}(0) = \emptyset.
$$
The topology $\T$ has as a subbasis of closed sets the set $\{\,
V_a, W_a \mid a \in L\,\}$, with $V_a=\{\,f \in \mph{\Lalg}{\twoB}
\mid f(a)=0 \,\}$ and $W_a=\{\,f \in \mph{\Lalg}{\twoB} \mid f(a)=1
\,\}$.

TiRS graphs were defined by the present authors in~\cite{CGH15} as
an abstraction of the graphs ${\rm D}^\flat(\Lalg)=
(\mph{\Lalg}{\twoB},E)$ obtained from Plo\v{s}\v{c}ica's duals of
bounded lattices $\Lalg$ by forgetting the topology.

For a graph $\X=(X,E)$ and $x \in X$, the sets  $\{\, y \in X \mid
(x,y) \in E \,\}$ and $\{\, y \in X \mid (y,x) \in E  \,\}$ were
denoted in~\cite{CGH15} by $xE$ and $Ex$ respectively.
We defined 
the conditions (S), (R) and (Ti) for any graph
$\X=(X,E)$ as follows:
\begin{enumerate}
\item[(S)] for every $x, y \in X$, if $x \neq y$ then $xE\neq yE$ or $Ex \neq Ey$;
\item[(R)]
\begin{itemize}
\item[(i)] for all $x, z \in X$, if $zE\subsetneq xE$ then $(z,x) \notin E$;
\item[(ii)] for all $y, z \in X$, if $Ez \subsetneq Ey$ then $(y,z) \notin E$;
\end{itemize}
\item[(Ti)] for all $x,y \in X$, if $(x,y) \in E$, then there exists $z \in X$ such that
$zE\subseteq xE$ and $Ez\subseteq Ey$.
\end{enumerate}

A \emph{TiRS graph} was in~\cite{CGH15} defined as a graph
$\X=(X,E)$ with a reflexive relation $E$ and satisfying the
conditions (R), (S) and (Ti).
For any bounded lattice $\Lalg$, its dual graph $\X={\rm D}^\flat(\Lalg)$ is a TiRS graph~\cite[Proposition 2.3]{CGH15}. 


We further recall that a \emph{frame} is a structure $(X_1, X_2,R)$,
where $X_1$ and $X_2$ are non-empty sets and $R\subseteq X_1 \times
X_2$. For an arbitrary frame $\mathbf{F}=(X_1, X_2, R)$ the
conditions (S) and (R) are defined as follows:

\begin{enumerate}
\item[(S)] for all $x_1, x_2 \in X_1$ and $y_1, y_2 \in  X_2$,
\begin{itemize}
\item[(i)] $x_1 \neq x_2$ implies $x_1 R \neq x_2 R$;
\item[(ii)] $y_1 \neq y_2$ implies $R y_1 \neq R y_2$.
\end{itemize}
\item[(R)]
\begin{itemize}
\item[(i)] for every $x \in X_1$ there exists $y\in  X_2$ such that $\neg(xRy)$ and
$\forall w \in X_1$\ $((w \neq x \, \And  x R \subseteq w R) \Rightarrow w R y)$;
\item[(ii)] for every $y\in  X_2$ there exists $x \in X_1$ such that $\neg(xRy)$ and $\forall z \in  X_2$\ $((z \neq y \, \And  R y \subseteq Rz) \Rightarrow x R z)$.
\end{itemize}
\end{enumerate}

The frames that satisfy the conditions (R) and (S) are called
reduced separated frames, or RS frames for short, and were
introduced by Gehrke \cite{Ge06} as a two-sorted generalisation of
Kripke frames to be used for relational semantics of substructural
logics.

The (Ti) condition introduced in~\cite{CGH15} for frames $(X_1,
X_2,R)$ was motivated by the (Ti) condition on graphs:
\begin{enumerate}
\item[(Ti)] for every $x \in X_1$ and for every $y\in  X_2$, if $\neg(x R y)$ then there exist $w \in X_1$ and $z\in  X_2$ such that
\begin{itemize}
\item[(i)]  $\neg(w R z)$;
\item[(ii)] $xR\subseteq wR$ and $Ry\subseteq Rz$;
\item[(iii)] for every $u \in X_1$, if  $u\neq w$ and $wR\subseteq uR$ then $uRz$;
\item[(iv)] for every $v\in X_2$, if $v\neq z$ and  $Rz\subseteq Rv$ then $wRv$.
\end{itemize}
\end{enumerate}

A \emph{TiRS frame} was in \cite{CGH15} defined as a frame $(X_1,
X_2,R)$ that satisfies conditions (R), (S) and (Ti), i.e. it is an
RS frame that satisfies condition (Ti). A one-to-one correspondence
between TiRS graphs and TiRS frames was then shown in \cite{CGH15}.
We recall here some facts of this correspondence that will be needed
in the next section.

\begin{df}[{\cite[Definition 2.5]{CGH15}}]\label{def:equivrel}
Let $\X = (X,E)$ be a graph. The \emph{associated frame} \emph{$\rho(\X)$} is the frame
$(X_1,X_2,R_{\rho(\X)})$ where
\begin{newlist}
\item[\upshape{(i)}] $X_1 = X/{\sim_1}$ for the equivalence relation $\sim_1$ on $X$  given by
\[x \sim_1 y \textrm{  if } xE = yE;\]
\item[\upshape{(ii)}] $X_2 = X/{\sim_2}$ for the equivalence relation $\sim_2$ on $X$  given by
\[x \sim_2 y \textrm{  if } Ex = Ey;\]
\item[\upshape{(iii)}] $R_{\rho(\X)}$ is the relation given by
\[  [x]_1 R_{\rho(\X)} [y]_2 \:  \Longleftrightarrow\: (x,y) \notin E, \] where $[x]_1$ and $[y]_2$ are, respectively, the $\sim_1$-equivalence
class of $x$ and the $\sim_2$-equivalence class of $y$.
\end{newlist}
We omit the subscript ${\rho(\X)}$ in $R_{\rho(\X)}$ whenever it is clear to which relation $R$ refers.
\end{df}



%


If \, $\X=(X,E)$ is a TiRS graph, then the associated frame\, $\rho(\X)=(X_1, X_2, R_{\rho(\X)})$ is a TiRS
frame~\cite[Proposition 2.6]{CGH15}.
Then it follows that if $\Lalg$ is a bounded lattice, $\X={\rm D}^\flat(\Lalg)$ is its dual TiRS graph and  $\rho({\rm D}^\flat(\Lalg))$ is the associated frame,
then $\rho({\rm D}^\flat(\Lalg))$ is a TiRS frame (cf. \cite[Corollary 2.7]{CGH15}).

\begin{df}[{\cite[Definition 2.8]{CGH15}}]\label{def:frame_to_graph1}
Let $\F=(X_1,X_2,R)$ be a TiRS frame. The \emph{associated graph}
$\gr(\F)$ is $(H_\F,K_\F)$ where the vertex set $H_\F$ is the subset of $X_1 \times X_2$
of all pairs $(x,y)$
that satisfy the following conditions:
\begin{newlist}
\item[\upshape{(i)}] $\neg(x R y)$,
\item[\upshape{(ii)}] for every $u \in X_1$, if $u\neq x$ and $xR\subseteq uR$ then $uRy$,
\item[\upshape{(iii)}] for every $v \in X_2$, if $v\neq y$ and $Ry\subseteq Rv$ then $xRv$.
\end{newlist}
and the edge set $K_\F$ is formed by the pairs $((x,y),(w,z))$ such that $\neg(xRz)$.

We omit the subscript $\F$ in $H_\F$ and in $K_\F$ whenever it is clear which vertex set and edge set we refer to.
\end{df}

%

In \cite[Proposition 2.10]{CGH15} we showed that if\, $\F=(X_1,X_2,R)$ is a TiRS frame, then its associated graph $\gr(\F)$ is a TiRS graph.

\begin{df}[{\cite[Definition 2.11]{CGH15}}]\label{def:iso}
Two graphs $\X=(X,E_X)$ and $\Y=(Y,E_Y)$ are \emph{isomorphic}
(denoted
$\X \simeq \Y$)
if there
exists a bijective map $\alpha \colon X \to Y$
such that
$$
\forall x_1, x_2 \in X
\quad (x_1,x_2) \in E_X
\iff (\alpha(x_1), \alpha(x_2)) \in E_Y
$$
and we refer to such a map as the graph-isomorphism $\alpha \colon \X \to \Y$.

Two
frames $\F=(X_1,X_2,R_F)$ and $\mathbf{G}=(Y_1,Y_2,R_G)$ are \emph{isomorphic}
(denoted $\F \simeq \mathbf{G}$)
if there exists
a pair $(\beta_1, \beta_2)$ of bijective maps
$\beta_i \colon X_i \to Y_i$ ($i=1,2$) with
$$
\forall x_1\in X_1\ \forall x_2 \in X_2
\quad
\big(x_1 R_F x_2
\iff \beta_1(x_1) R_G \beta_2(x_2)
\big)
$$
and we refer to such a pair as the
frame-isomorphism
$(\beta_1, \beta_2) \colon \F \to \G$.
\end{df}


Now for a TiRS graph $\X=(X,E)$, a map $\alpha_X \colon X \to gr(\rho(\X))$ is defined by $\alpha_X(x) = ([x]_1,[x]_2)$.
The next result shows that $\alpha_X$ is a graph isomorphism and that the correspondence
between TiRS graphs and TiRS frames is one-to-one.

\begin{thm}[{\cite[Theorem 2.13]{CGH15}}] \label{goodcorrespondence}
Let $\X=(X,E)$ be a TiRS graph and\, $\F=(X_1,X_2,R)$ be a TiRS frame. Then
\begin{newlist}
\item[\rm(a)] the graphs\, $\X$ and\, $\gr(\rho(\X))$ are isomorphic;
\item[\rm(b)] the frames\, $\F$ and\, $\rho(\gr(\F))$ are isomorphic.
\end{newlist}
\end{thm}

\section{TiRS graph and TiRS frame morphisms}\label{morph}

In this section we extend the one-to-one correspondence between TiRS graphs and TiRS frames from \cite{CGH15} into the full categorical framework. We start by defining
the concepts of TiRS graph and TiRS frame morphisms.

\begin{df}\label{def:graph-mor}
Let $\X = (X,E_X)$ and $\Y = (Y,E_Y)$ be TiRS graphs. A \emph{TiRS
graph morphism} is a map  $\varphi \colon X\to Y$ that satisfies the
following conditions:
\begin{newlist}
\item[(i)] for $x_1, x_2 \in X$, if $(x_1,x_2)\in E_X$ then $(\varphi(x_1),\varphi(x_2))\in E_Y$;
\item[(ii)] for $x_1, x_2 \in X$, if $x_1E_X\subseteq x_2E_X$ then $\varphi(x_1)E_Y \subseteq \varphi(x_2)E_Y$;
\item[(iii)] for $x_1, x_2 \in X$, if $E_Xx_1\subseteq E_Xx_2$ then $E_Y\varphi(x_1) \subseteq E_Y\varphi(x_2)$.
\end{newlist}
\end{df}

We note that every graph isomorphism and its inverse are TiRS graph
morphisms.

\begin{df}\label{def:frame-mor}
Let $\F = (X_1, X_2,R_F)$ and $\G = (Y_1, Y_2,R_G)$ be TiRS frames.
A~\emph{TiRS frame morphism} $\psi \colon \F \to \G$ is a a pair
$\psi=(\psi_1, \psi_2)$ of maps $\psi_1 \colon X_1 \to Y_1$ and
$\psi_2 \colon X_2 \to Y_2$ that satisfies the following conditions:
\begin{newlist}
\item[(i)] for $x\in X_1$ and $y\in X_2$, if $\psi_1(x)R_G\psi_2(y)$ then $xR_Fy$;
\item[(ii)] for $x,w \in X_1$, if $xR_F\subseteq wR_F$ then $\psi_1(x)R_G \subseteq \psi_1(w)R_G$;
\item[(iii)] for $y,z \in X_2$, if $R_Fy\subseteq R_Fz$ then $R_G\psi_2(y) \subseteq R_G\psi_2(z)$;
\item[(iv)] for $x\in X_1$ and $y\in X_2$, if $(x,y)\in H_\F$ then $(\psi_1(x), \psi_2(y))\in H_\G$.
\end{newlist}
\end{df}

We note that a frame isomorphism is a TiRS morphism.

Henceforth we shall refer to TiRS graph morphisms and to TiRS frame
morphisms simply as graph morphisms and frame morphisms
respectively.

Our main result in this section puts our one-to-one correspondence
between TiRS graphs and TiRS frames into a full categorical
framework. The last two statements are illustrated by the diagrams
in Fig.~\ref{fig:mor}.

\begin{thm}\label{thm:TiRSduality}
Let $\X = (X,E_X)$ and $\Y = (Y,E_Y)$ be TiRS graphs and let $\F =
(X_1, X_2,R_F)$ and $\G = (Y_1, Y_2,R_G)$ be TiRS frames.
\begin{itemize}
\item[(1)] If $\varphi \colon \X \to \Y$ is a TiRS graph morphism, then,  for $\rho(\varphi)_1 \colon X/\!\!\!
\sim_1 \to Y/\!\!\!\sim_1$ and $\rho(\varphi)_2 \colon
X/\!\!\!\sim_2 \to Y/\!\!\!\sim_2$ the maps defined by
$\rho(\varphi)_1([x]_1)=[\varphi(x)]_1$ and
$\rho(\varphi)_2([x]_2)=[\varphi(x)]_2$, for all $x\in X$, the pair
$\rho(\varphi)=(\rho(\varphi)_1,\rho(\varphi)_2)$ is a TiRS frame
morphism from $\rho(\X)$ to $\rho(\Y)$.
\item[(2)]
If the pair $\psi=(\psi_1,\psi_2) \colon \F \to \G$ is a TiRS frame
morphism, then the  map $\gr(\psi) \colon \gr(\F) \to \gr(\G)$
defined by $\gr(\psi)(x,y)=(\psi_1(x),\psi_2(y))$, for $(x,y) \in
H_F$, is a TiRS graph morphism.
\item[(3)] If $\varphi \colon \X \to \Y$ is a TiRS graph morphism, then $\gr(\rho(\varphi))\circ \alpha_X=\alpha_Y \circ \varphi$.
\item[(4)] If $\psi \colon \F \to \G$ is a TiRS frame morphism, then $\rho(\gr(\psi))\circ \beta_F=\beta_G \circ \psi$.
\end{itemize}
\end{thm}

\begin{proof}
(1) First we show that $\rho(\varphi)_1$ is well defined. Let $x, y
\in X$. If $[x]_1=[y]_1$ then $xE_X=yE_X$ which implies
$\varphi(x)E_Y=\varphi(y)E_Y$ and so
$[\varphi(x)]_1=[\varphi(y)]_1$, by the definition of a TiRS graph
morphism. Similarly we prove that $[\varphi(x)]_2=[\varphi(y)]_2$
whenever $[x]_2=[y]_2$. Next we prove that conditions (i) to (iv) of
the definition of a TiRS frame morphism are satisfied by
$\rho(\varphi)$. For (i), let $x,y \in X$ and assume
$\rho(\varphi)_1([x]_1)R_{\rho(\Y)}\rho(\varphi)_2([y]_2)$.
 Then $(\varphi(x),\varphi(y))\notin E_Y$ yielding
  that $(x,y)\notin E_X$ and so $[x]_1R_{\rho(\X)}[y]_2$.
For (ii), let $x, w \in X$. Then the following holds:
\begin{align*}
[x]_1R_{\rho(\X)}\subseteq [w]_1R_{\rho(\X)}&\iff w E_X \subseteq
xE_X
\Rightarrow \varphi(w)E_Y \subseteq \varphi(x)E_Y\\
& \iff [\varphi(x)]_1R_{\rho(\Y)}\subseteq
[\varphi(w)]_1R_{\rho(\Y)}.
\end{align*}
Hence (ii) is satisfied. Similarly we conclude that (iii) holds.
Finally (iv) follows from
\cite[Lemma 2.12]{CGH15} where we showed that for a TiRS graph $\X=(X,E)$, the elements of $H_{\rho(\X)}$ are exactly the pairs $([x]_1,[x]_2)$, with $x \in X$.

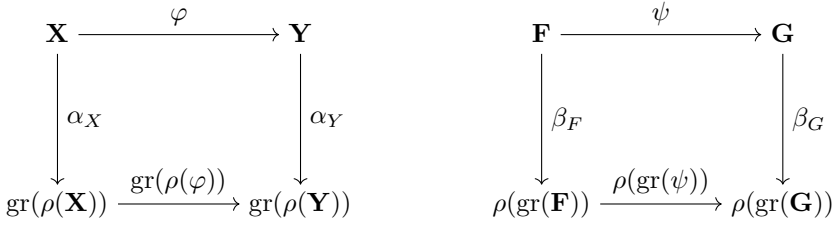
\begin{figure}
\begin{center}
\begin{tabular}{cc}
\begin{tikzpicture}
\matrix (m) [matrix of math nodes, row sep=2.5em, column sep=2.3em,
text  height=1.5ex, text depth=0.25ex]
{{\X} & & {\Y} & & {\F} & & {\G} \\
                    & &                 & &                 & & \\
{\gr(\rho(\X))}   &  &  {\gr(\rho(\Y))} & & {\rho(\gr(\F))}   &  &
{\rho(\gr(\G))} \\};

\path[->] (m-1-1) edge node[auto] {$\varphi$} (m-1-3); \path[->]
(m-1-1) edge node[auto] {$\alpha_X$} (m-3-1); \path[->] (m-1-3) edge
node[auto] {$\alpha_Y$} (m-3-3); \path[->] (m-3-1) edge node[auto]
{$\gr(\rho(\varphi))$} (m-3-3);

\path[->] (m-1-5) edge node[auto] {$\psi$} (m-1-7); \path[->]
(m-3-5) edge node[auto] {$\rho(\gr(\psi))$} (m-3-7); \path[->]
(m-1-7) edge node[auto] {$\beta_G$} (m-3-7); \path[->] (m-1-5) edge
node[auto] {$\beta_F$} (m-3-5);
\end{tikzpicture}
\end{tabular}
\end{center}
\caption{TiRS graph morphisms and TiRS frame morphisms
\label{fig:mor}}
\end{figure}

(2) First we note that condition (iv) of the definition of a TiRS
frame morphism satisfied by $\psi$  guarantees that the map
$\gr(\psi) $ is well defined. Next we prove that conditions (i) to
(iii) of the definition of a TiRS graph morphism are satisfied by
$\gr(\psi) $. Let $(x,y), (w,z)\in H_\F$. If $((x,y), (w,z))\in
K_\F$ then $\neg(xR_Fz)$ which implies
$\neg(\psi_1(x)R_\G\psi_2(z))$ and consequently $$(\gr(\psi)(x,y),
\gr(\psi)(w,z))=((\psi_1(x),\psi_2(y)), (\psi_1(w),\psi_2(z)))\in
K_\G.$$ Hence $(x,y), (w,z)$ satisfies (i). For (ii), we observe
that,
$$
(x,y)K_\F\subseteq (w,z)K_\F \iff wR_F\subseteq xR_F
$$
and
$$
\gr(\psi)(x,y)K_\G\subseteq \gr(\psi)(w,z)K_\G \iff
\psi_1(w)R_G\subseteq \psi_1(x)R_G,
$$
which follows from (iii) of \cite[Lemma 3.9]{CGH15}. As
$\psi$ is a TiRS morphism,
we also have
$$
wR_F\subseteq xR_F \Rightarrow \psi_1(w)R_G\subseteq \psi_1(x)R_G.
$$ Hence $(x,y), (w,z)$ satisfies (ii). Similarly we conclude that (iii) also holds.

(3) Let $x \in X$. We have that
\begin{align}
(\gr(\rho(\varphi))\circ \alpha_X)(x)&=\gr(\rho(\varphi))([x]_1,[x]_2)\notag\\
&=(\rho(\varphi)_1([x]_1), \rho(\varphi)_2([x]_2))\notag\\
&=([\varphi(x)]_1, [\varphi(x)]_2)\notag\\
&=(\alpha_Y\circ \varphi)(x).\notag
\end{align}

(4) Let $x \in X_1$. There exist $y\in X_2$ such that $(x,y)\in
H_\F$. We have that
\begin{align}
(\rho(\gr(\psi))\circ \beta_F)(x)&=\rho(\gr(\psi)([(x,y)]_1)\notag\\
&=[\gr(\psi)(x,y)]_1\notag\\
&=[(\psi_1(x),\psi_2(y))]_1,\notag
\end{align}
where $\psi_1 \colon X_1 \to Y_1$ and $\psi_2 \colon X_2 \to Y_2$
satisfy $\psi=(\psi_1,\psi_2)$. Since ${(x,y)\in H_\F}$ and $\psi$
is a TiRS morphism, we also have $(\psi_1(x),\psi_2(y))\in H_\G$ and
so
$$[(\psi_1(x),\psi_2(y))]_1=\beta_G(\psi(x))=
(\beta_G \circ \psi) (x). $$
\end{proof}

\begin{cor}\label{cor:equivalence}
The category of TiRS graphs with TiRS graph morphisms is equivalent
to the category of  TiRS frames with TiRS frame morphisms via the
functors given by $\rho$ and $\gr$ as described above.
\end{cor}


\bigskip

There are other definitions of morphisms between two frames (or contexts) $\mathbf{F}=(X_1,X_2,R_F)$ and
$\mathbf{G}=(Y_1,Y_2,R_G)$ that are used in the literature.
Deiters and Ern\'{e}~\cite{DE09} use a pair of maps $(\alpha,\beta)$ where $\alpha \colon X_1 \to Y_1$ and $\beta \colon X_2 \to Y_2$
as we do above.
Gehrke~\cite[Section 3]{Ge06} uses a pair of relations $(R,S)$ where $R\subseteq X_2 \times Y_1$ and $S \subseteq X_1 \times Y_2$.
More recently, Moshier~\cite{Mosh12} (see also Jipsen~\cite{Jip12}) defined a context morphism to
be a single relation $S\subseteq X_1 \times Y_2$.

\section{Perfect lattices dual to TiRS structures}

Consider a complete lattice $\mathbf{C}$ and let
$\mathbf{F}(\mathbf{C})=(\jty(\C),\mty(\C),\leqslant)$ where
$\jty(\C)$ and $\mty(\C)$ denote the sets of
completely join-irreducible and completely
meet-irreducible elements of
$\C$, respectively. We will refer to $\mathbf{F}(\C)$ as the frame
coming from $\C$.
For the opposite direction, consider an
RS frame
$\mathbb{F}=(X,Y,R)$. For $A\subseteq X$ and $B \subseteq Y$, let
$$\RTR(A)=\{\, y \in Y \mid (\forall a \in A)(aRy)\,\} \ \text{and}\ \RTL(B)=\{\, x \in X \mid (\forall b \in B)(xRb)\,\}.$$
Now consider the complete lattice of Galois-closed sets
(ordered by inclusion):
$$\mathcal{G}(\mathbb{F})=\{\, A \subseteq X \mid A=(\RTL \circ \RTR)(A)\,\}.$$
By results from Gehrke~\cite[Section 2]{Ge06} we know that the completely  join-irreducible elements and completely meet-irreducible elements of
$\mathcal{G}(\mathbb{F})$ are identified as follows:
$$\jty(\mathcal{G}(\mathbb{F}))=\{\, (\RTL \circ\RTR)(\{x\}) \mid x \in X\,\} \ \text{and} \
\mty(\mathcal{G}(\mathbb{F}))=\{\, Ry \mid y \in Y\,\}.$$

Below we introduce a condition that refines the class of perfect lattices. At the end of this section we will conclude that
every perfect lattice that is
the canonical extension of some bounded lattice will have this property.

\begin{df} A perfect lattice satisfies the condition (PTi) if for all $x\in \jty(\C)$ and for all $y \in \mty(\C)$, if
$ x\nleq y$ then there exist $w \in \jty(\C)$, $z \in \mty(\C)$ such that
\begin{enumerate}[(i)]
\item $w \leq x$ and $y \leq z$
\item $w \nleq z$
\item $(\forall u \in \jty(\C))(u < w \Rightarrow u \leq z)$
\item $(\forall v \in \mty(\C))(y < v \Rightarrow w \leq v)$
\end{enumerate}
\end{df}

In Fig.~\ref{fig:PTi} we give a pictorial depiction of the (PTi)
condition. We have indicated the sets ${\uparrow}x$, ${\uparrow}w$,
${\downarrow}y$ and ${\downarrow}z$.
We see that the (PTi) condition for $\C$ essentially starts with an
arbitrary disjoint filter-ideal pair $({\uparrow}x,{\downarrow}y)$
generated by elements $x\in \jty(\C)$ and $y \in \mty(\C)$. It says
that every such disjoint filter-ideal pair is contained in a maximal
disjoint filter-ideal pair $({\uparrow}w,{\downarrow}z)$ where again
$w\in \jty(\C)$ and
$z \in \mty(\C)$
(here maximality is understood such that
neither of
${\uparrow}w$ and ${\downarrow}z$ can be
enlarged
without breaking their disjointness).

\begin{figure}
\begin{center}
\begin{tikzpicture}
\draw (0,0) ellipse (1cm and 2cm);
\smallpo{0,2}
\node at (0.25,2.15) {$1$};
\smallpo{0,-2}
\node at (0.25,-2.15) {$0$};
\smallpo{0.5,0}
\node at (0.7,0) {$x$};
\li{(0.866,1)--(0.5,0)--(-0.661,1.5)}
\smallpo{-0.3,-1}
\node at (-0.1,-1) {$y$};
\li{(-0.661,-1.5)--(-0.3,-1)--(0.484,-1.75)}
\smallpo{0.5,-0.7}
\node at (0.5,-0.4) {$w$};
\li{(-0.866,1)--(0.5,-0.7)--(1,0)}
\smallpo{-0.7,0.3}
\node at (-0.9,0.3) {$z$};
\li{(-1,0)--(-0.7,0.3)--(0.745,-1.33)}
\end{tikzpicture}
\caption{The (PTi) condition illustrated. \label{fig:PTi}}
\end{center}
\end{figure}
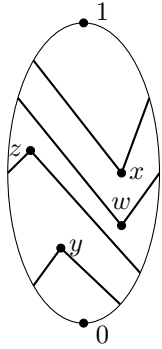

\begin{lem}\label{lem:PTi-to-Ti} Let $\C$ be a perfect lattice. If\, $\C$ satisfies {\upshape(PTi)} then the RS frame
$\mathbf{F}(\C)=(\jty(\C),\mty(\C),\leq)$ satisfies {\upshape(Ti)}.
\end{lem}
\begin{proof} First observe that when translating the condition (Ti) from a general RS frame to $\mathbb{F}(\C)$ we have that
$xR={\uparrow}x$ and $Ry={\downarrow}y$. The fact that $\mathbb{F}(\C)$ satisfies (Ti)
follows then
from the fact that $u < w$ implies $u\neq w$ and ${\uparrow}w \subsetneq {\uparrow}u$.
\end{proof}

We want to characterise the condition (PTi) on the Galois closed sets arising from an RS frame $\mathbb{F}=(X,Y,R)$.
The following lemma will assist us in this task.
\begin{lem} \label{lem:R-upsets} Consider the RS frame $\mathbb{F}=(X,Y,R)$. Then
\begin{newlist}
\item[{\upshape (i)}] $w \in (\RTL\circ\RTR)(\{x\})$ if and only if $xR \subseteq wR$;
\item[{\upshape (ii)}] $(\RTL\circ\RTR)(\{w\})\subseteq (\RTL\circ\RTR)(\{x\})$ if and only if $xR\subseteq wR$;
\item[{\upshape (iii)}] $(\RTL\circ\RTR)(\{x\})\subseteq Ry$ if and only if $xRy$.
\end{newlist}
\end{lem}
\begin{proof} For (i) we have
\begin{align*}
w \in (\RTL\circ\RTR)(\{x\}) \quad &\Leftrightarrow \quad (\forall z \in \RTR(\{x\}))(wRz) \\
&\Leftrightarrow \quad (\forall z \in Y)(xRz \Rightarrow wRz) \\
&\Leftrightarrow \quad xR \subseteq wR.
\end{align*}
To assist with the proof of (ii), note that $\RTR(\{x\})=Rx$ and $\RTR(\{w\})=Rw$. If we assume that $xR \subseteq wR$ then the fact that
$\RTL: \powerset(Y) \to \powerset(X)$ is order-reversing gives us that $(\RTL\circ\RTR)(\{w\})\subseteq (\RTL\circ\RTR)(\{x\})$.
For the converse, if $(\RTL\circ\RTR)(\{w\})\subseteq (\RTL\circ\RTR)(\{x\})$ then since
$\RTR:\powerset(X)\to\powerset(Y)$ is order-reversing and since
$(\RTR\circ\RTL\circ\RTR)(\{w\})=\RTR(\{w\})$, we get $xR \subseteq wR$.
The statement (iii) is exactly~\cite[Proposition 2.6]{Ge06}.
\end{proof}

We want to prove that when an RS frame $\mathbb{F}=(X,Y,R)$ satisfies the (Ti) condition, the perfect lattice of Galois-closed sets
$\mathcal{G}(\mathbb{F})$ satisfies (PTi). In order to make the proof easier to follow, it will be useful to translate the condition (PTi) from
the setting of a general perfect lattice to the setting of $\mathcal{G}(\mathbb{F})$.

\begin{lem} Let $\mathbb{F}=(X,Y,R)$ be an RS frame.
Assume that the following set of conditions is satisfied by $\mathcal{G}(\mathbb{F})$:

For all $x \in X$ and all $y\in Y$, if $\neg(xRy)$ then there exist $p\in X$, $q \in Y$ such that
\begin{newlist}
\item[{\upshape(i)}] $xR \subseteq pR$ and $Ry \subseteq Rq$
\item[{\upshape(ii)}] $\neg(pRq)$
\item[{\upshape(iii)}] $(\forall u \in X)(pR \subsetneq uR \Rightarrow uRq)$
\item[{\upshape(iv)}] $(\forall v \in Y)(Rq \subsetneq Rv \Rightarrow pRv)$
\end{newlist}
\end{lem}
Then the lattice $\mathcal{G}(\mathbb{F})$ satisfies (PTi).
\begin{proof}
This follows using Lemma~\ref{lem:R-upsets} to translate (PTi) conditions to the complete lattice $\mathcal{G}(\mathbb{F})$.
\end{proof}

\begin{lem} \label{lem:Ti-to-PTi} Let $\mathbb{F}=(X,Y,R)$ be an RS frame. If\, $\mathbb{F}$ satisfies {\upshape(Ti)} then $\mathcal{G}(\mathbb{F})$ satisfies
{\upshape(PTi)}.
\end{lem}
\begin{proof}

Let $\mathbb{F}=(X,Y,R)$ be an RS frame satisfying {\upshape(Ti)} (i.e. a TiRS frame). Take arbitrary $x \in X$ and $y\in Y$ and assume that $\neg(xRy)$.
In the perfect lattice $\mathcal{G}(\mathbb{F})$ coming from $\mathbb{F}$ consider the sets $A=(\RTL\circ \RTR)(\{x\})$ and $B=Ry$.
Then $A \in\jty(\mathcal{G}(\mathbb{F}))$, $B \in \mty(\mathcal{G}(\mathbb{F})$ and $A \nsubseteq B$ using Lemma~\ref{lem:R-upsets}(iii).

We have that
\begin{align*}
(\RTL\circ\RTR)(\{x\}) \nsubseteq Ry \quad &\Rightarrow \quad (\exists w \in X)(xR \subseteq wR \:\&\: \neg(wRy))\\
&\Rightarrow
\quad (\exists w \in X)\Big[ xR \subseteq wR \:\&\: \\
&\qquad \:\: (\exists p \in X)(\exists q \in Y)\Big(\neg(pRq) \:\&\: wR\subseteq pR \:\&\:Ry \subseteq Rq \\
&\qquad \quad \:\&\: (\forall u \in X)(pR\subsetneq uR \Rightarrow uRq )  \\
&\qquad \qquad \:\&\:(\forall v \in Y)(Rq \subsetneq Rv \Rightarrow pRv) \Big)\Big] \\
\end{align*}
The only part of the (PTi) condition for $\mathcal{G}(\mathbb{F})$ that is not now immediate is the fact that we need
$xR \subseteq pR$. This follows from the $xR \subseteq wR\subseteq pR$ and the transitivity of set containment.
\end{proof}

Now we are ready to show that the canonical extensions of lattices
are PTi lattices and so they indeed are `more' than just perfect
lattices. For this we cite our final result from~\cite{CGH15}:

\begin{prop}[{\cite[Corollary 3.11]{CGH15}}]\label{P3-3.11}
Let\, $\Lalg$ be a bounded lattice and\, $\X ={\rm D}^\flat(\Lalg)$
be its dual TiRS graph. Let $\rho(\X)$ be the frame associated to\,
$\X$ and\, $\mathrm{G}(\rho(\X))$ be its corresponding perfect
lattice of Galois-closed sets.

The lattice\, $\mathrm{G}(\rho(\X))$ is the canonical extension of\,
$\Lalg$.
\end{prop}

The result can be illustrated by the diagram in
Fig.~\ref{fig:tandem}. The given bounded lattice $\Lalg$ is firstly
assigned its Plo\v{s}\v{c}ica dual space $\D{\Lalg} =
(\mph{\Lalg}{\twoB},E,\T)$, and then the Plo\v{s}\v{c}ica dual graph
$\X ={\rm D}^\flat(\Lalg)= (\mph{\Lalg}{\twoB},E)$ is obtained by
forgetting the topology. This is a TiRS graph and so the frame
$\rho(\X)$ associated to\, $\X$ in our one-to-one correspondence
developed in \cite{CGH15} between TiRS graphs and TiRS frames is a
TiRS frame. Hence by Lemma~\ref{lem:Ti-to-PTi} above, the perfect
lattice $\mathrm{G}(\rho(\X))$ of Galois-closed sets corresponding
in Gehrke's representation to the frame $\rho(\X)$ is a PTi
lattice. By Proposition~\ref{P3-3.11}, the lattice
$\mathrm{G}(\rho(\X))$ is the canonical extension of the given
lattice $\Lalg$.

\begin{figure}  [ht]
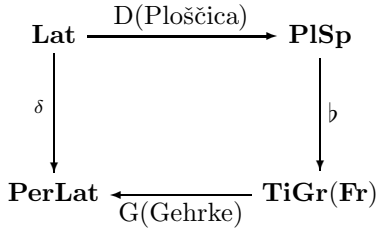

\begin{diagram}[labelstyle=\textstyle]
\mathbf{Lat}& \rTo^{{\mathrm D} (\text{Plo\v{s}\v{c}ica})}
& \mathbf{PlSp}\\
\dTo^{^\delta} & & \dTo_{\flat}\\
\mathbf{PerLat}&  \lTo_{{\mathrm G} (\text{Gehrke})} & \mathbf{TiGr}
(\mathbf{Fr})
\end{diagram}
\caption{Plo\v{s}\v{c}ica and Gehrke in tandem. \label{fig:tandem}}
\end{figure}

Hence we have our final result of this section:

\begin{thm}\label{thm:CE-is-PTi}
The canonical extension of
a
bounded lattice is a PTi
lattice. 
\end{thm}

Gehrke
and Vosmaer~\cite{GV11} showed that the canonical extension of a lattice need not be meet-continuous, and hence need not always be
algebraic.
Theorem~\ref{thm:CE-is-PTi} gives us further information about the structure of canonical extensions of bounded lattices.

\section{Examples}

Our goal in this section is to
illustrate
that the PTi condition adds to the current description of the
canonical extension of a bounded lattice.
We focus on non-distributive examples. Canonical extensions of distributive lattices are known to be completely distributive complete lattices.
To show that our new condition does indeed add to the current
description, we give an
example of a perfect lattice that is not PTi. Giving an example of a
PTi lattice that is not the canonical extension of a lattice would
be the same as giving an example of a TiRS graph that is not of the
form $(\mph{\Lalg}{\twoB},E)$ for some bounded lattice $\Lalg$.
Hence this is the same as the representable TiRS graph
(representable poset) problem.

Our goals are:
\begin{enumerate}
\item Give an example of a complete
non-distributive
lattice which is a PTi lattice but is not the canonical extension of any bounded lattice.
\item Give an example of a perfect
non-distributive
lattice that is not a PTi lattice. 
\end{enumerate}

\begin{figure}[ht]
\begin{center}
\begin{tikzpicture}[scale=1]

\node (p0) at (1,0) {$p_0$};
\node (p1) at (1,0.8) {$p_1$};
\node (p2) at (1,1.6) {$p_2$};
\node (p3) at (1,2.4) {};
\draw [->>] (p0) edge (p1);
\draw [->>] (p1) edge (p2);
\draw [->>] (p2) edge (p3);

\dotli{(1,2.4)--(1,2.8)}
\dotli{(1,3.2)--(1,3.6)}

\node (q0) at (1,6) {$q_0$};
\node (q1) at (1,5.2) {$q_1$};
\node (q2) at (1,4.4) {$q_2$};
\node (q3) at (1,3.6) {};
\draw [->>] (q1) edge (q0);
\draw [->>] (q2) edge (q1);
\draw [->>] (q3) edge (q2);

\node (k) at (-2,3) {$k$};
\draw [->] (k) edge (q1);
\draw [->] (k) edge (q2);
\draw [->] (k) edge (p0);
\draw [->] (k) edge (p1);
\draw [->] (k) edge (p2);

\draw [->] (k) edge (p3);
\draw [->] (k) edge (q3);

%
%

\po{5.75,2.5} \po{7,0.45} \po{7,0.95} \po{7,1.45} \po{7,2.5}
\po{7,3.55} \po{7,4.05} \po{7,4.55}

\dotli{(7,0.95)--(7,2.5)} \dotli{(7,4.05)--(7,2.5)}
\li{(7,0.45)--(5.75,2.5)--(7,4.55)--(7,4.05)}
\li{(7,0.45)--(7,0.95)--(7,1.45)} \li{(7,3.55)--(7,4.05)--(7,4.55)}

\node at (5.15,2.5) {$z=y$};
\node at (7.4,4.05) {$x$};
\node at (7.4,2.5) {$m$};
\node at (7.4,0.95) {$w$};
\node at (7,0.05) {$A_L$};

\end{tikzpicture}
\caption{A TiRS graph that is not the graph of MPH's of any bounded lattice (left) and its dual PTi lattice $A_L$ that is not a
canonical extension (right). The double-headed arrows on the graph emphasize that transitivity holds amongst the vertical edges. \label{fig:A_L}}
\end{center}
\end{figure}
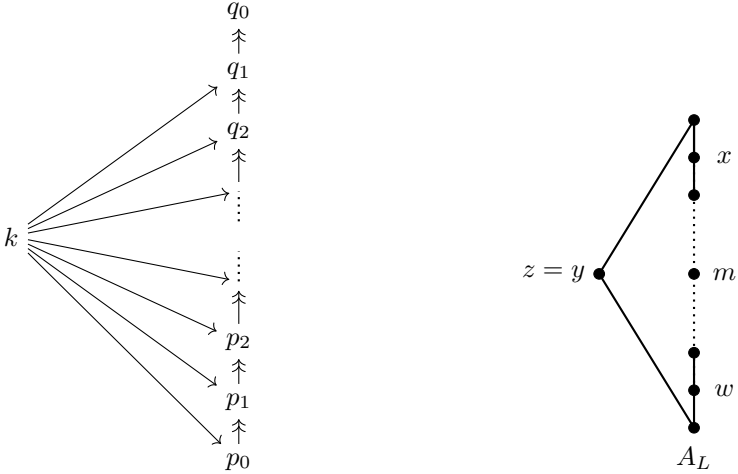

\begin{ex}\label{A_L}
Consider the complete lattice $A_L$
depicted on the right in~Fig.~\ref{fig:A_L}.
We will denote by $m$, the middle element of the infinite chain $\omega \oplus \mathbf{1} \oplus \omega^{\partial}$
and $y$ is above the bottom and below the top but incomparable with all other elements.

The TiRS graph dual to $A_L$ is $X= \{p_i \mid i \in \omega\} \cup \{q_j \mid j \in \omega\} \cup \{k\}$
with the relation $E$ given by
$$p_0 < p_1 < p_2 < \ldots < p_n < p_{n+1} < \ldots < q_{n+1} < q_n < q_{n-1} \ldots <q_1 < q_0$$
$$ \cup \{ (k,p_i)\mid i \in \omega\} \cup \{ (k,q_j) \mid j \geqslant 1 \}$$
(it is depicted on the left in~Fig.~\ref{fig:A_L}).
To be clear, the $p_i$'s and $q_j$'s form a poset (it is transitive) that is order-isomorphic to $\omega \oplus \omega^{\partial}$ while the
element $k$ is related to everything except the top of the chain.

Recall that MPE's are ordered by: $\varphi \leqslant \psi$ if and only if $\varphi^{-1}(1)\subseteq \psi^{-1}(1)$.
%
The MPE's $\varphi_1$ and $\varphi_0$ are defined by
$\varphi_1(x)=1$ and $\varphi_0(x)=0$ for all $x \in X$. 
All but one of the other MPE's have $k \mapsto 0$ and then they split the chain at some point.
When $\varphi$ splits the chain by sending the $p_i$'s to $0$ and the $q_j$'s to $1$ then you get the limit point in the middle of $A_L$. The interesting MPE is the map does the following for $a \in X$:
$$\varphi(a) = \begin{cases} 1 & \text{ if } a = k \\ 0 & \text{ if } a = q_0 \\ - &\text{ otherwise} \end{cases}$$
This interesting MPE is the incomparable point that makes $A_L$ non-distributive.

It is quite easy to show that the lattice $A_L$ is a PTi lattice; we indicated on the right in~Fig.~\ref{fig:A_L} what the elements $w \in \jty(A_L)$ and $z \in \mty(A_L)$
are for the chosen elements $x\in \jty(A_L)$ and $y \in \mty(A_L)$. The fact that the lattice $A_L$ is not the canonical extension of any bounded lattice is harder to show and it
follows from Proposition~\ref{prop:notCEbutPTi} below.
\end{ex}

\begin{prop}\label{prop:notCEbutPTi}
There is no bounded lattice $\Lalg$ and lattice embedding $e \colon \Lalg \to A_L$ such that
$(e,A_L)$ is the canonical extension of $\Lalg$.
\end{prop}
\begin{proof}
Suppose there are no bounded lattice $\Lalg$ and an embedding $e
\colon \Lalg \to A_L$ such that $(e,A_L)$ is the canonical extension
of $\Lalg$. Clearly the top and bottom element of $A_L$ are,
respectively $e(1)$ and $e(0)$ where $1$ and $0$ are the top and
bottom element of $\Lalg$. Now consider the set of elements
$(A_L)\setminus \{e(0),e(1),m\}$. It is easy to see that each of
these elements is completely join-irreducible in $A_L$ and hence we
have, by~Lemma~\ref{GH-Lem3.4}, that each of these elements is the
meet of the embedding of a 
filter of $L$. Hence each of
the elements of $(A_L)\setminus \{e(0),e(1),m\}$ is a filter
element. Dually, it is easy to see that each element of
$(A_L)\setminus \{e(0),e(1),m\}$ is completely meet-irreducible and
again by~Lemma~\ref{GH-Lem3.4} they are all the join of the
embedding of an 
ideal of $L$ and hence are all ideal
elements. Thus every element of $(A_L)\setminus \{e(0),e(1),m\}$ is
both ideal and filter and hence must be of the form $e(a)$ for some
$a \in L$. Now consider the element $m$. Since $m = \bigwedge
\omega^{\partial}$, and since every element of $\omega^{\partial}$
is the image of an element of $L$ under $e$, we have that $m$ is a
filter element of $A_L$. Also, $m = \bigvee \omega$ and every
element of $\omega$ is the image of an element of $L$ under $e$.
Therefore $m$ is also an ideal element of $A_L$. Hence $m$ must be
of the form $e(b)$ for some $b \in L$. Thus we have that $\Lalg
\cong A_L$ and that the embedding $e$ is a bijection.

Now we show that $(e,A_L)$ cannot be the canonical extension of $\Lalg$. Observe that since $m= \bigwedge \omega^{\partial}=\bigvee \omega$ we
have that $\bigwedge \omega^{\partial} \leqslant \bigvee \omega$. However, for any finite subset $A'\subseteq \omega^{\partial}$ and
any finite subset $B'\subseteq \omega$ we will have $\bigvee B' < \bigwedge A'$. Hence $(e,A_L)$ is not a compact completion
of $\Lalg$.
\end{proof}

\begin{figure}
\begin{center}
\begin{tikzpicture}[scale=1]

\node (a1) at (-7,2.5) {$a_1$};
\node (a0) at (-6,2.5) {$a_0$};
\node (a2) at (-5,2.5) {$a_2$};
\node (a3) at (-4,2.5) {$a_3$};
\dotli{(-3.5,2.5)--(-2,2.5)}

\node (b0) at (-7,1.5) {$b_0$};
\node (b1) at (-6,1.5) {$b_1$};
\node (b2) at (-5,1.5) {$b_2$};
\node (b3) at (-4,1.5) {$b_3$};
\dotli{(-3.5,1.5)--(-2,1.5)}

\draw [->] (a1) edge (b0);
\draw [->] (a0) edge (b1);
\draw [->] (a2) edge (b1);
\draw [->] (a3) edge (b1);
\draw [->] (a2) edge (b2);
\draw [->] (a3) edge (b2);
\draw [->] (a3) edge (b3);

\po{2,0.45} \po{0.75,3} \po{2,2.75} \po{2,3.25} \po{2,3.75}

\dotli{(2,0.5)--(2,2.45)}
\li{(2,0.45)--(0.75,3)--(2,3.75)--(2,3.25)--(2,2.75)--(2,2.5)}

\node at (0.35,3) {$y$};
\node at (2.4,3.75) {$x_0$};
\node at (2.4,3.25) {$x_1$};
\node at (2.4,2.75) {$x_2$};
\node at (3,1.75) {$\omega^{\partial}$};
\node at (2,0) {$M_L$};

%
%

\end{tikzpicture}
\caption{An RS frame that is not Ti (left) and its dual perfect lattice that is not PTi (right). \label{fig:M_L}}
\end{center}
\end{figure}
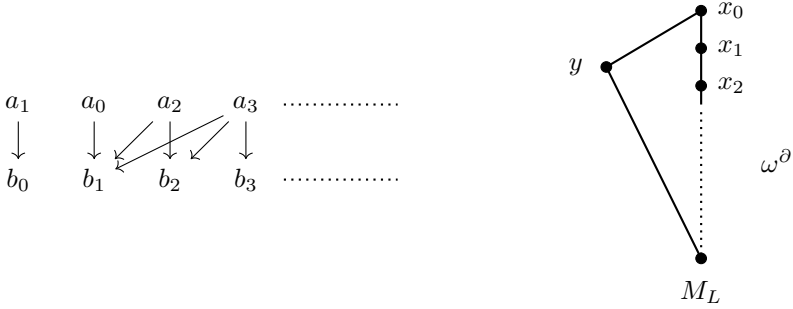

\begin{ex}
We consider the complete lattice $M_L$
depicted on the right in~Fig.~\ref{fig:M_L}. The order is given by the poset $\mathbf{1}\oplus \omega$ with an additional element $y$ incomparable to all elements except the top and the bottom.
It can easily be seen that $M_L$ is a perfect lattice
($J^{\infty}(M_L)=M^{\infty}(M_L)=
\{y\} \cup \{\, x_i \mid i \geqslant 1 \,\}$).  It
is not PTi since there are no $w$ and $z$ for the pair
$x_j \nleqslant y$ ($j \geqslant 1$).

The RS frame corresponding to it was already mentioned in~\cite[page~128]{CGH15} as an example of
an RS frame which is not TiRS (it is indicated on the left in~Fig.~\ref{fig:M_L}): Let $X_1=\{a_i\}_{i \in \omega}$,
$X_2=\{b_i\}_{i \in \omega}$ and let
$$R=\{(a_1,b_0),(a_0,b_1)\}\cup \{\,(a_i,b_j) \mid 2\le i, 1 \le j \le i\,\}.$$
By considering
$\neg(a_0 R b_0)$ it is rather straightforward to show that $(X_1,X_2,R)$ does not satisfy (Ti).
\end{ex}

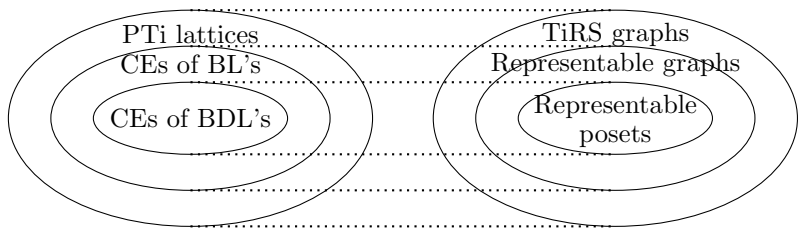
\begin{figure}[ht]
\begin{center}
\begin{tikzpicture}[scale=0.8]
    \draw (0,0) ellipse (1.6cm and 0.6cm);
        \draw (0,0) ellipse (2.3cm and 1.2cm);
        \draw (0,0) ellipse (3cm and 1.8cm);
        \node at (0,0) {CEs of BDL's};
        \node at (0,0.9) {CEs of BL's};
        \node at (0,1.4) {PTi lattices};
        \dotli{(0,0.6)--(7,0.6)}
        \dotli{(0,-0.6)--(7,-0.6)}
        \dotli{(0,1.2)--(7,1.2)}
        \dotli{(0,1.8)--(7,1.8)}
        \dotli{(0,-1.2)--(7,-1.2)}
        \dotli{(0,-1.8)--(7,-1.8)}
        \draw (7,0) ellipse (1.6cm and 0.6cm);
    \draw (7,0) ellipse (2.3cm and 1.2cm);
    \draw (7,0) ellipse (3cm and 1.8cm);
    \node at (7,0.2) {Representable};
    \node at (7,-0.3) {posets};
    \node at (7,0.9) {Representable graphs};
    \node at (7,1.4) {TiRS graphs};
\end{tikzpicture}
\caption{The correspondences between classes of PTi lattices and classes of TiRS graphs. \label{fig:ellipses}}
\end{center}
\end{figure}

Our final picture Fig.~\ref{fig:ellipses} describes the correspondence between PTi lattices and TiRS graphs and between their important subclasses: (i) the canonical
extensions of bounded lattices inside the PTi lattices and representable graphs (as dual graphs of bounded lattices) inside the TiRS graphs; (ii) the canonical
extensions of bounded distributive lattices inside the canonical extensions of bounded lattices and representable posets (as dual graphs of bounded distributive
lattices) inside the representable graphs.

A natural question that we asked already in~\cite[pages~126--127]{CGH15} was which TiRS graphs arise as duals of bounded lattices.
In the case of bounded distributive lattices (denoted as BDL's in Fig.~\ref{fig:ellipses}) this question reduces to the
question of which posets are \emph{representable posets} which seems to be extremely hard. Examples of non-representable posets are also examples of non-representable
graphs as any poset is automatically a TiRS graph. We mention an example of a non-representable poset due to
Tan~\cite{Tan74}
from the 1970s: $T := \omega \oplus \omega^{\delta}$.
The perfect lattice corresponding to this TiRS graph is the PTi lattice ${T_L} := \omega \oplus \mathbf{1} \oplus
\omega^{\delta}$.

\section*{Acknowledgements}

The first author gratefully acknowledges the hospitality
of Matej Bel University during his research visit in September 2017.
%
The third author acknowledges support from Slovak grant VEGA 1/0337/16 and the hospitality of the University of Lisbon during his visit in September 2019. 


\end{document}